\documentclass{article}
\usepackage{amsmath}
\usepackage{amsthm}
\usepackage{amssymb}
\usepackage{verbatim}
\usepackage{graphicx}
\usepackage{float}
\usepackage{geometry}
\usepackage{indentfirst}

\newtheorem{theorem}{Theorem}
\newtheorem{lemma}{Lemma}
\newtheorem{proposition}{Proposition}

\title{Twisting in One-Dimensional Periodic Vlasov-Poisson System}
\author{Sangwook Tae}

\numberwithin{equation}{section}
\numberwithin{theorem}{section}
\numberwithin{lemma}{section}
\numberwithin{proposition}{section}

\begin{document}
	\maketitle
	\begin{abstract}
		We prove that twisting and filamentation occur near a family of stable steady states for a one-dimensional periodic Vlasov-Poisson system, which describes the electron dynamics under a fixed ion background. More precisely, we establish the growth in time of the $L^1$ norm of the gradient for the electron distribution function and the corresponding flow map in the phase space. To support this result, we prove the existence results of stable steady states for a class of ion densities on the torus.
	\end{abstract}
	\section{Introduction}
	\subsection{Main Results}
	In this paper, we consider the one-dimensional Vlasov-Poisson system on $\mathbb{T}=\mathbb{R}/\mathbb{Z}$ with a given nonnegative background charge density $\rho$:
	\begin{flalign}
	\label{Vlasov}
		\begin{split}
			&\partial_{t}f(x, v, t)+v\cdot\partial_{x}f(x, v, t)-E(x,t)\cdot\partial_{v}f(x, v, t)=0\\
			&E(x, t)=\int_{0}^{1}K(x, y)\left[\rho(y)-\int_{-\infty}^{\infty}f(y, v, t)dv\right]dy.
		\end{split}
	\end{flalign}
	\indent
	Here, the kernel $K$ is defined as\footnote{$K$ serves as a gradient of Green's function for the Laplacian on a torus.}
	\begin{flalign*}
		K(x, y)=\begin{cases}
			y, & \text{for } 0\leq x< y\\
			y-1, & \text{for } y< x\leq 1
		\end{cases}
	\end{flalign*}
	and we shall assume that $$\int_{0}^{1}\rho(y)dy=1=\int_{\mathbb{T}\times\mathbb{R}} f(x, v, 0)dxdv=\int_{0}^{1}\rho_e(y, t)dy$$
	with $\rho_e$ being the time-dependent electron density function. This equation describes the distribution of electrons $f$ in the phase space $\mathbb{T}\times\mathbb{R}$ whose motion is driven by the electric field $E$ induced by electrons themselves and surrounding positive ions in the background \cite{Majda}, which is applicable to study the motion of the plasma in heavy ions so that the positions of ions remain fixed. Specifically, the electric field satisfies the one-dimensional Poisson equation, which implies the kernel representation as above. We note that since we are working on the periodic domain, the total charge of the electrons and ions must balance each other. \\
	\indent
	Our main results establish that at least linear in time filamentation for $f$ occurs near a
	family of stable steady states for (\ref{Vlasov}). Filamentation for the distribution function, which roughly refers to the creation of long and thin structures, is a generic phenomenon clearly observed in various numerical simulations of the Vlasov–Poisson equation \cite{Fijalkow}, \cite{Sandburg}, \cite{Colombi}. More specifically, numerical computations show that as the solution evolves in time, strong filamentation occurs at least in some region of the phase space, homogenizing the electron density function to some degree.\\
	\indent
	Following the work \cite{Drivas}, we quantify filamentation by the growth of the norm $\Vert \nabla f \Vert_{L^1(\mathbb{T}\times\mathbb{R})}$. This is appropriate, since upon conservation of all $L^p$ norms for $f$, growth of the gradient in $L^1$ implies formation of long and thin structures in an $O(1)$-region of the phase space.\\
	\indent
	We note that linear in time filamentation is a simple consequence when Landau damping occurs for the solution, which corresponds to the extreme situation where mixing in the phase space occurs and makes the density uniform over the spatial domain. However, we emphasize that for Landau damping to occur, stringent conditions both on the equilibrium state and its perturbations, such as Gevrey regularity, are required \cite{Gagnebin}, \cite{Ionescu}, and \cite{Mouhot}. Counterexamples exist for the Sobolev perturbation \cite{Bedrossian}, and the absence of Landau damping for more general steady states and perturbations is clear from numerical simulations \cite{Sandburg}.\\
	\indent
	Filamentation is obtained as a consequence of the “twisting phenomenon,” which can be briefly explained as follows. Consider the shear flow on the strip $\mathbb{T}\times [0,1]$ described below\footnote{This type of flow occurs in completely integrable Hamiltonian systems.}:
	\begin{flalign*}
		\frac{d\theta}{dt}=J\quad \text{and} \quad \frac{dJ}{dt}=0.
	\end{flalign*}
	The flow map of this system is
	$$\Phi_t(\theta, J)=(\theta+Jt, J).$$
	\indent
	Thus, we note that if we lift the dynamics to the universal cover $\mathbb{R}\times[0, 1]$, we can see that ``filamentation" occurs, that is, for an arbitrary bounded set $A$, $\operatorname{diam}(\Phi_t(A))$ grows linearly in time. As a consequence, for generic $f:\mathbb{T}\times[0, 1]\to \mathbb{R}$, the $L^1$ norm of $\nabla \left(f\circ\Phi_t^{-1}\right)$ grows linearly in time.\footnote{It means that the quantity conserved along the flow experiences filamentation. Examples can be vorticity in thes 2D Euler equation, and so on.} This phenomenon was first named as twisting in \cite{Drivas}.\\
	\indent
	It is reasonable to expect that this phenomena is stable under small perturbations in the flow map, that is, small perturbations in the Lagrangian sense. A remarkable fact shown in \cite{Drivas} is that this phenomenon is stable even in small perturbation in the Eulerian sense, that is, it is stable under small non-autonomous perturbations in the velocity vector field (or stream function). The smallness is measured only in an integral sense. Moreover, this is a phenomenon that even occurs locally. That is, if the twisting occurs on an invariant subset of the whole phase space, it is robust under small non-autonomous perturbations in this invariant subset. Such perturbations do not have to leave this subset invariant. This result is stated in section \ref{Thm} in detail. \\
	\indent
	The result introduced above was shown to be applicable in many situations, such as Arnold diffusion, gradient growth of vorticity in SQG equations, and so on \cite{Drivas}. In this paper, we apply this result to one-dimensional Vlasov-Poisson system. The twisting occurs in two-dimensional phase space $\mathbb{T}\times\mathbb{R}$. It seems at first that the result in \cite{Drivas} is not directly applicable since the phase space is noncompact. However, due to the locality of this result, it is still applicable if we focus on the compact invariant subset.\\
	\indent
	Now, we move on to the  details. What we have proved is that the twisting occurs when the initial data is sufficiently near the stationary state of the following form.\footnote{It seems at first that this assumption is a big restriction. However, for all stationary solutions, since the Hamiltonian $\varphi\left(\frac{1}{2}v^2+U^*(x)\right)$ is conserved along the characteristics, the distribution function is constant along each connected component of energy levels. Thus, we see that this is a relatively mild assumption.}
	\begin{flalign}
		\label{Stationary}
		f^*(x, v)=\varphi\left(\frac{1}{2}v^2+U^*(x)\right)
	\end{flalign}
	where 
	\begin{flalign}
		U^*(x)=\int_{0}^{x}\int_{0}^{1}K(z, y)\left[\rho(y)-\int_{-\infty}^{\infty}f^{*}(y, v)dv\right]dydz.
	\end{flalign}
	Here, we assume that 
	\begin{enumerate}
		\item $\varphi\in C^1(\mathbb{R})$
		\item $\varphi$ is supported on $(-\infty, E_{max}]$
		\item $\varphi'(E)<0$ for $E<E_{max}$
	\end{enumerate}
	We state the main theorem below.
	\begin{theorem}
		\label{Main}
		Suppose that $f^*(x, v)=\varphi\left(\frac{1}{2}v^2+U^*(x)\right)$ where $
		\varphi$ satisfies the above assumption. Then, for all $C_1>0$, there exists $\delta>0$ depending on $C_1$ such that if $f(x, v, 0)\leq C_1$ and $\int_{\mathbb{T}}\int_{\mathbb{R}} (1+v^2)\vert f(x, v, 0)-f^*(x, v) \vert dvdx <\delta$
		holds, the twisting occurs in the following sense.\\
		For the characteristic flow $\Phi$ defined as in Section \ref{P}, the $L^1$ norm of its gradient grows linearly in time:
		$$\exists A>0 \quad \text{s.t.} \quad \Vert \nabla\Phi_t \Vert_{L^1(\mathbb{T}\times [-v_0, v_0])}\geq At.$$
	\end{theorem} 
	\indent
	By adding (minor) additional conditions to $f$ near $f^*$, we actually have more. In particular, we have the following $L^1$ gradient growth for $f$:
	\begin{theorem}
		\label{Corr}
		Assume that the conditions for $f^*$ in Theorem 1.1 are fulfilled. In addition, suppose that there exist two distinct level sets of $f(\cdot, \cdot, 0)$ with different values (say, 0 and 1) contained in $M=\operatorname{supp}(f^*)$ such that they enclose areas sufficiently close to the area of $M$. Then, for all $C_1>0$, there exists $\epsilon>0$ depending on $C_1$ such that if $$\vert f(x, v, 0)\vert\leq C_1$$ and $$\int_{\mathbb{T}}\int_{\mathbb{R}} (1+v^2)\vert f(x, v, 0)-f^*(x, v) \vert dvdx<\epsilon$$
		then the $L^1$ norm of $\nabla f$ grows linearly in time:
		$$\exists B>0 \quad \text{s.t.} \quad \Vert \nabla f(\cdot, \cdot, t) \Vert_{L^1(\mathbb{T}\times [-v_0, v_0])}\geq Bt.$$
		Here, $v_0>0$ is chosen such that $M\subset \mathbb{T}\times[-v_0, v_0]$. 
	\end{theorem}
	This result seems to be the first infinite gradient growth result for smooth solutions of the Vlasov-Poisson equation, except for the case when Landau damping is proved.\footnote{However, in \cite{patch_1} there are related results for the electron sheet case, where the electron distribution is supported in a regular curve in the two-dimensional configuration space. Namely, the folding of the electron sheets is observed, and it may be somehow related to this situation due to the fact that both are related to the small-scale creations.}\\
	\indent
	These theorems, however, assume that a stable steady distribution of such forms exists, which does not seem to be available for non-constant background ion density in $\mathbb{T}$, to the best of our knowledge. Thus, later in this paper, we prove the existence of such steady states in some special cases. 
	\subsection{Literature}
	\subsubsection{Vlasov Poisson Equation}
	The Vlasov-Poisson equation has been studied extensively. For smooth solutions, the local well posedness in $\mathbb{R}^d$, and even global well-posedness for the cases $d=1, 2$ were proven in \cite{Horst}. The global well-posedness for $d=3$ for general initial data was proven in \cite{K.P.}. These results can also be found in \cite{Glassey}. By a slight modification of the proof, it is possible to improve this result to the case in $\mathbb{T}^d$ in the presence of background
	ion density $\rho$. However, it was also shown in \cite{Horst2} that there is a finite time blow-up for $d\geq4$.\\
	\indent
	There are also lots of results concerning the existence of stable steady states. The central result in this problem is the method based on constructing the conserved quantity which is continuous and positive definite so-called the Energy-Casimir method \cite{Casimir}. Many results used this method. For example, in \cite{Yan&Gerhard_2}, Yan Guo and Gerhard Rein showed the existence of rotationally symmetric steady states stable under perturbations with the same symmetry. Gerhard Rein improved this result by removing the symmetric assumptions on perturbations in \cite{G.R}. In \cite{G.W}, G.Wolansky constructed symmetric stable steady states for three-dimensional Vlasov-Poisson system in the stellar case. Also, in \cite{Y.G}, Yan Guo constructed so-called ``Camm" type stable steady states in the same situation. There are more results concerning the stellar case, such as \cite{Jack}, \cite{Yan&Gerhard}, and \cite{Yieh-Hei}. \\
	\indent
	In the case of charged particles, Peter Braasch, Gerhard Rein, and Jesenko Vukadinovich  constructed the stationary steady states in 3 dimensions using the Energy-Casimir method in \cite{Peter}. They constructed the steady states without the stationary ion approximations. In \cite{Maria}, Maria J. Caceres, Jose A. Carrillo, and Jean  Dolbeault constructed $L^p$ stable states in the case of electrons moving in an external electric field again, by using the Energy-Casimir method. \\
	\indent
	Finally, there are some previous results related to twisting phenomena. In \cite{patch_1}, Ryszard Stanislaw Dziurzynski proved the folding formation of electron patches. The setting in that paper considers the patch solution, which is not smooth, but corresponds to the gradient growth in our settings. Also, in \cite{patch_2}, Emeric Roulley proved the emergence of traveling periodic electron layers. At first glance, it seems to contradict our result since there is no growth but rather periodicity. However, our results give twisting for generic perturbations, not for all. Furthermore, \cite{patch_2} considers the patch case while we consider smooth solutions. 
	\subsubsection{Comparison With 2D Euler Equation}
	There are corresponding results for the 2D Euler equation on a torus. Namely, in \cite{Denisov}, Denisov has proven the growth of the gradient of vorticity in time. However, the strategy of the proof in this paper is essentially different. The proof of Denisov's result obtains the gradient growth 
	of the vorticity by obtaining the $L^\infty$ stability of the velocity fields, which relies on the stability of the $L^p$ norm of the vorticity. However, there are no corresponding results available in the Vlasov-Poisson equation. The characteristic velocity field is given as $(v, -\nabla U)$ and, its $L^\infty$ stability is not known. However, we were able to obtain the $L^1$ stability, which was enough for applying the twisting theorem in \cite{Drivas}.
	\subsection{Outline of the Paper}
	The rest of the paper is organized as follows. In Section 2, we list and explain some important tools for proving the main theorem. In Section 3, we obtain an estimate for the stream function in terms of the distribution function $f$. In Section 4, we prove the stability of $f$ near the equilibrium state $f^*$. We use these two results to prove Theorem \ref{Main} and Theorem \ref{Corr}, which are done in Section 5. Finally, we obtain the existence of a stationary distribution $f^*$ in some special cases, including uniform ion charge cases and ion distribution close to the "Delta" distribution.
	
	\section{Preliminaries}
	\subsection{Characteristics and Stream Functions}
	\label{P}
	It is well known that the distribution function is conserved along the characteristic flow: i.e., the following holds.
	\begin{flalign}
		f(x, v, t)=f(\Phi_t^{-1}(x, v), 0)
	\end{flalign}
	where $\Phi$ is the flow corresponding to the following differential equations
	\begin{flalign}
		\frac{dx}{dt}=v\quad \text{and}\quad \frac{dv}{dt}=-E(x, t)
	\end{flalign}
	satisfying $\Phi_0=\operatorname{id}$.
	Here, $E(x, t)$ is the same one as in (\ref{Vlasov}). 
	This characteristic equation is a (non-autonomous) Hamiltonian system with Hamiltonian $$\psi(x, v, t)=\frac{1}{2}v^2+U(x, t)=\frac{1}{2}v^2+\int_{0}^{x}E(y, t)dy.$$
	As a consequence, this flow is (Lebesgue) measure-preserving.
	\subsection{Conservation of Energy}
	Another main ingredient in this paper is the conservation of energy. The energy of the state $f$ is defined as follows:
	$$H(f)=\int_{\mathbb{T}}\int_{\mathbb{R}} \frac{v^2}{2}f(x, v, t) dvdx+\int_{\mathbb{T}} \frac{1}{2} E(x, t)^2 dx.$$
	For sufficiently smooth solutions, it is known that this quantity is conserved in time:
	$$H(f(\cdot, \cdot, t))=H(f(\cdot, \cdot, 0)).$$
	
	\subsection{Twisting Theorem in Hamiltonian Flows}
	\label{Thm}
	We now state the theorem of stability of twisting in \cite{Drivas} here. We state the result in more simpler form here. The full statement can be found in the original paper \cite{Drivas}: 
	\begin{theorem}
		\label{Twisting}
		Let $M$ be a two-dimensional annular domain (domain diffeomorphic to an annulus), and $A\subset M$ be a subdomain. Suppose that $\psi_*$ is a smooth time-independent stream function on M which has two periodic orbits in $A$ such that
		\begin{enumerate}
			\item They are noncontractible in $A$.
			\item They have different periods.
		\end{enumerate}
		Then, the following holds:
		There exists $\epsilon>0$ depending only on $\psi_{*}\vert_A$ such that for all smooth $\psi:M\times\mathbb{R}\to \mathbb{R}$ such that $u=\nabla^\perp\psi$ is uniformly bounded and $\Vert \psi_{*}-\psi \Vert_{W^{1,1}}<
		\epsilon$ for all $t\in\mathbb{R}$, we have the following:
		\begin{enumerate}
			\item For any lift $\widetilde{\Phi_t}$ of the flow $\Phi$ to the universal cover $\widetilde{M}$, the diameter $\tilde{\Phi_t}(M)\subset\widetilde{M}$ grows linearly in time.
			\item $\Vert \nabla\Phi_t \Vert_{L^1}$ grows linearly in time as $\vert t\vert\to\infty$. 
		\end{enumerate} 
	\end{theorem}
	\subsection{Strategy of the Proof}
	We prove the main theorem in the following strategy. First, we obtain an estimate of the difference of stream functions in terms of their distribution functions. Then, we prove the $L^2$ stability of the stationary distribution function $f^{*}$ by the Energy-Casimir method. We finally combine the two results and use Theorem \ref{Twisting} to conclude. 
	\section{Estimates for Stream Function in Terms of Distribution $f$}
	In order to apply the theorem, we need to estimate the following quantity:
	$$\Vert \psi-\psi^* \Vert_{W^{1, 1}\left(\mathbb{T}\times[-v_0, v_0]\right)}$$
	where $\psi$ and $\psi^*$ are stream functions corresponding to $f$ and $f^*$, respectively. We have the following Lemma.
	\begin{lemma}
		\label{ineq_1}
		Let $f^*$ and $f$ be a solution to (\ref{Vlasov}). Let $f^*$ be a stationary solution satisfying (\ref{Stationary})\footnote{Of course, we assume that the assumptions on $\varphi$ stated previously hold.}. Then, for all $u_{0}>0$, the following inequality holds.
		\begin{flalign*}
			&\Vert \psi-\psi^* \Vert_{W^{1,1}(\mathbb{T}\times[-v_0, v_0])}\leq 4v_0(2u_0)^{\frac{1}{2}} \Vert f-f^* \Vert_{L^2(\mathbb{T}\times\mathbb{R})} +\frac{8v_0}{u_{0}^2}\left( H(f)+H(f^*) \right).
		\end{flalign*}
	\end{lemma}
	\begin{proof}
		We first estimate the $L^{1}$ norm of $f-f^*$ as follows:
		\begin{flalign*}
			&\Vert \psi-\psi^* \Vert_{W^{1,1}(\mathbb{T}\times[-v_0, v_0])}=\int_{-v_0}^{v_0}\int_{\mathbb{T}}\vert U(x, t)-U^*(x) \vert dxdv+\int_{-v_0}^{v_0}\int_{\mathbb{T}}\vert E(x, t)-E^*(x) \vert dxdv\\
			&=2v_{0}\left(\int_{\mathbb{T}} \left\vert \int_{0}^{x}\int_{\mathbb{T}} K(z, y)\int_{-\infty}^{\infty} (f(y, v, t)-f^*(y, v))dvdydz \right\vert dx\right)\\
			&\quad+2v_{0}\left(\int_{\mathbb{T}} \left\vert \int_{\mathbb{T}} K(x, y)\int_{-\infty}^{\infty} (f(y, v, t)-f^*(y, v))dvdy \right\vert dx\right)\\
			&\leq 4v_0 \left( \int_{\mathbb{T}}\int_{-\infty}^{\infty} \left\vert f(x, v, t)-f^*(x, v) \right\vert dvdx \right)=4v_{0} \Vert f-f^* \Vert _{L^1(\mathbb{T}\times\mathbb{R})}.
		\end{flalign*}\\
		\indent
		Next, we bound $\Vert f-f^* \Vert _{L^1(\mathbb{T}\times\mathbb{R})}$ in terms of $L^{2}$-norm of $f-f^*$ and the total energy of $f$ and $f^*$. This is done by the following Chebyshev argument:
		\begin{flalign*}
			&\Vert f-f^* \Vert _{L^1(\mathbb{T}\times\mathbb{R})}=\int_{\mathbb{T}}\int_{-u_0}^{u_0} \vert f(x, v, t)-f^*(x, v) \vert dvdx +\int_{\mathbb{T}}\int_{\vert v \vert\geq u_0}\vert f(x, v, t)-f^*(x, v) \vert dvdx \\&\leq \left( 2u_0 \int_{\mathbb{T}}\int_{-u_0}^{u_0} \vert f(x, v, t)-f^*(x, v) \vert^2 dvdx  \right)^{\frac{1}{2}}+\frac{2}{u_0^2} \int_{\mathbb{T}}\int_{\vert v \vert\geq u_0} \frac{v^2}{2} \left( f(x, v, t)+f^*(x, v) \right) dvdx\\
			&\leq (2u_0)^{\frac{1}{2}} \Vert f-f^* \Vert_{L^2(\mathbb{T}\times\mathbb{R})} +\frac{2}{u_{0}^2}\left( H(f)+H(f^*) \right). 
		\end{flalign*}
		Combining the two inequalities gives the desired result.
	\end{proof}
	\section{Stability of $f$ near $f^*$}
	Having the bound of $W^{1, 1}$ norm of $\psi-\psi^*$ in terms of $\Vert f-f^* \Vert_{L^2(\mathbb{T}\times\mathbb{R})}$, we need to bound $\Vert f-f^* \Vert_{L^2(\mathbb{T}\times\mathbb{R})}$ in terms of initial conditions in order to prove the $W^{1, 1}$ stability of the stream function. This can be done by the Energy-Casimir method. We essentially follow \cite{Casimir}.
	More precisely, define the Casimir operator by
	\begin{flalign}
		C(f):=\int_{\mathbb{T}}\int_{\mathbb{R}} \Lambda(f(x, v))dvdx
	\end{flalign}
	where we define $\Lambda$ as follows.\\
	\indent
	Let $E_{min}=\inf_{x\in\mathbb{T}}U^*(x)$. Then, $$\varphi: [E_{min}, E_{max}]\to [0, \varphi(E_{min})]=[0, \varphi_{max}]$$ is bijective. \\
	\indent 
	Thus, $$\varphi^{-1}:[0, \varphi_{max}]\to [E_{min}, E_{max}]$$ is well defined. Define
	\begin{equation}
		\Lambda(\zeta)=-\int_{0}^{\zeta}\varphi^{-1}(\eta)d\eta,\qquad \zeta\in [0, \varphi_{max}]
	\end{equation}  
	and extend $\Lambda$ to $[0, \infty)$ by setting $$\Lambda(\zeta)=\Lambda(\varphi_{max})+\Lambda'(\varphi_{max})(\zeta-\varphi_{max})+\frac{1}{2}\Lambda''(\varphi_{max})\left(\zeta-\varphi_{max}\right)^2$$ for $\zeta>\varphi_{max}$ so that $\Lambda\in C^2((0, \infty))$.
	
	Since $f$ is constant along the measure preserving flow $\Phi$ described above, we see that this is a conserved quantity for the Valsov-Poisson system. Moreover, it is shown in \cite{Casimir} that 
	the following holds:\footnote{In \cite{Casimir}, these were proven in three-dimensional case, but the proof is the same for one-dimensional case.}
	\begin{enumerate}
		\item $f^*$ is a critical point of $H_C(f):=H(f)+C(f)$.
		\item There exists $c', c>0$ depending on $f^*$ such that the following holds:
		\begin{equation}
			 H_C(f)-H_C(f^*)\geq c\Vert f-f^* \Vert_{L^2(\mathbb{T}\times\mathbb{R})}^{2}
		\end{equation}
	\end{enumerate}
	In addition, we prove:
	\begin{proposition}
		Let $C_1>0$ be an arbitrarily large constant. Then, there exists a constant $C_2>0$ depending on $C_1$ and $f^*$ such that for all $f\leq C_1$ the following inequality holds:
		\begin{flalign*}
			\left\vert H_C(f)-H_C(f^*) \right\vert\leq C_2 \left( \int_{\mathbb{T}}\int_{\mathbb{R}} (1+v^2)\vert f-f^* \vert dvdx  \right).
		\end{flalign*}
	\end{proposition}
	\begin{proof}
		\begin{flalign*}
			&\vert H_C(f)-H_C(f^*) \vert\\
			&\leq \frac{1}{2}\int_{\mathbb{T}}\int_{\mathbb{R}} v^2 \vert f-f^*\vert dvdx+\frac{1}{2}\int_{\mathbb{T}} \vert E(x)^2-E^*(x)^2\vert dx+\int_{\mathbb{T}}\int_{\mathbb{R}}\vert\Lambda(f)-\Lambda(f^*)\vert dvdx\\
			&\leq \frac{1}{2}\int_{\mathbb{T}}\int_{\mathbb{R}} v^2 \vert f-f^*\vert dvdx+\frac{1}{2}\int_{\mathbb{T}} \vert E(x)-E^*(x)\vert\vert E(x)+E^*(x) \vert dx\\
			&\quad+\int_{\mathbb{T}}\int_{\mathbb{R}} L \vert f-f^*\vert dvdx \cdots(*)	
		\end{flalign*}
		Note that
		\begin{flalign*} 
			\vert E(x)\vert =\left\vert\int_{0}^{1}K(x, y)\left[\rho(y)-\int_{-\infty}^{\infty}f(y, v)dv\right]dy\right\vert\leq\int_{0}^{1}\left[\rho(y)+\int_{-\infty}^{\infty}f(y, v, t)dv\right]dy= 2
		\end{flalign*}
		and similarly, $\vert E^*(x) \vert\leq 2$.\\
		\indent
		Thus, we have
		\begin{flalign*}
			(*)&\leq \frac{1}{2}\int_{\mathbb{T}}\int_{\mathbb{R}} v^2 \vert f-f^*\vert dvdx+2\int_{\mathbb{T}} \vert E(x)-E^*(x)\vert dx+\int_{\mathbb{T}}\int_{\mathbb{R}} L \vert f-f^*\vert dvdx\\
			&\leq C_2 \left( \int_{\mathbb{T}}\int_{\mathbb{R}}(1+v^2)\vert f(x, v)-f^*(x, v) \vert dvdx \right).
		\end{flalign*}
	\end{proof}
	With these results, using that $H_C$ is a conserved quantity, we have for all solutions $f$ for (\ref{Vlasov}) satisfying $f(x, v, 0)\leq C_1$,
	\begin{flalign}
		\begin{split}
			\label{ineq_2}
			\Vert f(\cdot, \cdot, t)-f^*(\cdot, \cdot) \Vert_{L^2(\mathbb{T}\times\mathbb{R})}^2&\leq \frac{1}{c} \left\vert H_C(f(\cdot, \cdot, t))-H_C(f^*(\cdot, \cdot)) \right\vert=\frac{1}{c} \left\vert H_C(f(\cdot, \cdot, 0))-H_C(f^*(\cdot, \cdot)) \right\vert\\
			&\leq \frac{C_2}{c}  \int_{\mathbb{T}}\int_{\mathbb{R}} (1+v^2)\vert f(x, v, 0)-f^*(x, v) \vert dvdx.
		\end{split}
	\end{flalign}
	\section{Proof of Theorem \ref{Main} and Theorem \ref{Corr}}
	\subsection{Proof of Theorem \ref{Main}}
	We now prove Theorem \ref{Main}. Using Lemma \ref{ineq_1} and equation (\ref{ineq_2}), we obtain
	\begin{flalign*}
		&\Vert \psi(\cdot, \cdot, t)-\psi^*(\cdot, \cdot) \Vert_{W^{1,1}(\mathbb{T}\times[-v_0, v_0])}\\&\leq 4v_0(2u_0)^{\frac{1}{2}} \Vert f(\cdot, \cdot, t)-f^*(\cdot, \cdot) \Vert_{L^2(\mathbb{T}\times\mathbb{R})}+\frac{8v_0}{u_{0}^2}\left( H(f(\cdot, \cdot, t))+H(f^*(\cdot, \cdot)) \right)\\
		&\leq 4v_0(2u_0)^{\frac{1}{2}}\left[\frac{C_2}{c}  \int_{\mathbb{T}}\int_{\mathbb{R}} (1+v^2)\vert f(x, v, 0)-f^*(x, v) \vert dvdx \right]^{\frac{1}{2}} +\frac{8v_0}{u_{0}^2}\left( H(f(\cdot, \cdot, 0))+H(f^*(\cdot, \cdot)) \right)\\
		&\leq 4v_0(2u_0)^{\frac{1}{2}}\left[\frac{C_2}{c}  \int_{\mathbb{T}}\int_{\mathbb{R}} (1+v^2)\vert f(x, v, 0)-f^*(x, v) \vert dvdx \right]^{\frac{1}{2}} +\frac{8v_0}{u_{0}^2}\left( 2H(f^*(\cdot, \cdot))+\left\vert H(f(\cdot, \cdot, 0))-H(f^*(\cdot, \cdot)) \right\vert \right)\\
		&\leq 4v_0(2u_0)^{\frac{1}{2}}\left[\frac{C_2}{c}  \int_{\mathbb{T}}\int_{\mathbb{R}} (1+v^2)\vert f(x, v, 0)-f^*(x, v) \vert dvdx \right]^{\frac{1}{2}}\\
		&\quad+\frac{8v_0}{u_{0}^2}\left( 2H(f^*(\cdot, \cdot))+4\right)+ \frac{8 v_0}{u_0^2} \left( \int_{\mathbb{T}}\int_{\mathbb{R}} (1+v^2)\vert f(x, v, 0)-f^*(x, v) \vert dvdx  \right).\\
	\end{flalign*}
	\indent
	Now, pick $\varepsilon>0$. Then, by choosing $u_0$ large enough, we obtain
	$$\frac{8v_0}{u_{0}^2}\left( 2H(f^*(\cdot, \cdot))+4\right)< \frac{\varepsilon}{2}$$
	\indent
	Then, for chosen fixed $u_0$, we can choose $\delta>0$ such that 
	\begin{flalign*}
		&4v_0(2u_0)^{\frac{1}{2}}\left[\frac{C_2}{c}  \int_{\mathbb{T}}\int_{\mathbb{R}} (1+v^2)\vert f(x, v, 0)-f^*(x, v) \vert dvdx \right]^{\frac{1}{2}}+ \frac{8 v_0}{u_0^2} \left( \int_{\mathbb{T}}\int_{\mathbb{R}} (1+v^2)\vert f(x, v, 0)-f^*(x, v) \vert dvdx  \right)< \frac{\varepsilon}{2}
	\end{flalign*}
	proviede that $\int_{\mathbb{T}}\int_{\mathbb{R}} (1+v^2)\vert f(x, v, 0)-f^*(x, v) \vert dvdx <\delta$.
	Thus, we have that whenever $$\int_{\mathbb{T}}\int_{\mathbb{R}} (1+v^2)\vert f(x, v, 0)-f^*(x, v) \vert dvdx <\delta$$ holds, the inequality $$\Vert \psi(\cdot, \cdot, t)-\psi^*(\cdot, \cdot) \Vert_{W^{1,1}(\mathbb{T}\times[-v_0, v_0])}< \epsilon$$ is satisfied. By Theorem \ref{Twisting}, the twisting occurs for such $f$.
	\subsection{Proof of Theorem \ref{Corr}}
	The proof essentially follows that of Theorem 4.5 in \cite{Drivas}. The biggest difference is that in our case, the phase space is noncompact, so that we have to prevent the flow escaping from $M=\operatorname{supp}(f^*)$ too much. To do this, we have to use the fact that the $L^2$ mass is initially concentrated in $M$ and is stable, so that it is impossible to lose the $L^2$ mass in $M$ too much.  \\
	\indent
	We now start the proof. Let $\epsilon<\delta$ be sufficiently small, where $\delta$ is the one found in Theorem \ref{Main}, and let $\eta>0$ be a small number which will be specified later. Since $f^*(x, v)=\varphi\left(\frac{1}{2}v^2+U^*(x)\right)$, by assumptions on function $\varphi$, there exists $\rho>0$ such that
		$$\left\vert S_1 \right\vert=\left\vert \left\{ (x, v) \,:\,0<f^*(x, v, 0)<\rho \right\} \right\vert<\frac{\eta}{10}$$ 
	where $\left\vert A \right\vert$ denotes the Lebesgue measure of $A$.\\
	\indent
	Next, by the fact that the Casimir operator $C$ and $H$ are invariant under the flow and the properties of $H_C$, we have the following:
	\begin{flalign*}
		\left\Vert f(\cdot, \cdot, t)-f^*\right\Vert_{L^2(\mathbb{T}\times{R})}&\leq \frac{1}{c} \vert H_C(f(\cdot, \cdot, t))-H_C(f^*)\vert=\frac{1}{c} \vert H_C(f(\cdot, \cdot, 0))-H_C(f^*)\vert\\&\leq \frac{C_2}{c}  \int_{\mathbb{T}}\int_{\mathbb{R}} (1+v^2)\vert f(x, v, 0)-f^*(x, v) \vert dvdx
		\leq\frac{C_2\epsilon}{c}.
	\end{flalign*}
	\indent
	Now, let $S_2=\left\{ (x, v)\subset M\,:\, \Phi_t^{-1}((x, v))\notin M \right\}$. We want to obtain an upper bound on $\vert S_2\vert$. We consider two parts: $S_1\cap S_2$, and $S_2\backslash S_1$.
	The first one is obtained as follows:
	$$\vert S_1\cap S_2\vert \leq \vert S_1\vert <\frac{\eta}{10}.$$
	To obtain the estimate on $\vert S_2\backslash S_1\vert$, we use the mass argument as follows.\\
	\indent
	First, we note that by the definition of $S_1$ and $S_2$, for all $(x, v)\in S_2\backslash S_1$, we have that $f^*(x, v)\geq \rho$. Therefore, we have that
	\begin{flalign*}
		\frac{C_2\epsilon}{c}&>\left( \int_{S_2\backslash S_1} \vert f(x, v, t)-f^*(x, v) \vert^2 dxdv \right)^{\frac{1}{2}}\\&\geq \left( \int_{S_2\backslash S_1} \vert f^*(x, v)\vert^2 dxdv \right)^{\frac{1}{2}}-\left( \int_{S_2\backslash S_1} \vert f(x, v, t) \vert^2 dxdv \right)^{\frac{1}{2}}\\&> \rho\vert S_2\backslash S_1 \vert^{\frac{1}{2}}-\left( \int_{\Phi_t^{-1}\left(S_2\backslash S_1\right)} \vert f(x, v, 0) \vert^2 dxdv \right)^{\frac{1}{2}}\\&\geq\rho\vert S_2\backslash S_1 \vert^{\frac{1}{2}}-\left( \int_{(\mathbb{T}\times\mathbb{R})\backslash M} \vert f(x, v, 0) \vert^2 dxdv \right)^{\frac{1}{2}} \geq \rho\vert S_2\backslash S_1 \vert^{\frac{1}{2}}-\frac{C_2\epsilon}{c}.
	\end{flalign*}
	\indent
	Thus, we have $$\left\vert S_3\backslash (S_1\cup S_2) \right\vert <\left( \frac{2C_2\epsilon}{\rho c} \right)^2<\frac{\eta}{10}$$ if we further replace $\epsilon$ by a smaller number if necessary.\\
	\indent
	Gathering these results, we see that $\vert S_3\vert<\frac{\eta}{5}$. This ensures the  ability to control the amount(area) escaping $M$ at an arbitrary time $t>0$.\\
	\indent
	Finally, as in \cite{Drivas}, let us consider the universal cover $\mathbb{R}\times\mathbb{R}$ of $\mathbb{T}\times\mathbb{R}$ and lift the dynamics to that cover. Let $\widetilde{M}$ be the inverse image of $M$ with respect to the projection mapping. It is shown in the proof of Lemma 2.9 in \cite{Drivas} that there exists $\mu, C>0$, and two sets $A(t),\, B(t)\subset \widetilde{M}$ such that $$\vert A(t)\vert, \, \vert B(t)\vert >\mu$$ and $$\operatorname{dist}(A(t), B(t))\geq Ct$$
	\indent
	We now pick $\eta=\mu$. Then, by the previous estimate on $\vert S_3\vert$, we have that
	$$\vert \Phi_t^{-1}(A(t))\cap M \vert>\frac{4\eta}{5}\quad \text{and}\quad\vert \Phi_t^{-1}(B(t))\cap M \vert>\frac{4\eta}{5}. $$
	\indent
	By assumption, there exist two distinct level sets of $f(\cdot, \cdot, 0)$ with different values contained in $M=\operatorname{supp}(f^*)$ such that they enclose areas sufficiently close to the area of $M$. If the remaining areas are smaller than $\frac{\eta}{10}$, we see that these level curves meet both $\Phi_t^{-1}(A(t))$ and $\Phi_t^{-1}(B(t))$. From this stage, we follow the proof of Theorem 4.5 in \cite{Drivas}, which we briefly sketch here.\\
	\indent
	Let us denote the level sets mentioned by $\gamma_1$ and $\gamma_2$. Then, since each $\gamma_1$ and $\gamma_2$ intersect with both $\Phi_t^{-1}(A(t))$ and $\Phi_t^{-1}(B(t))$, we see that if we let $\widetilde{\Phi}_t$ be the lift of the flow map $\Phi_t$ to the cover, diameters of $\widetilde{\Phi}_t(\gamma_1)$ and $\widetilde{\Phi}_t(\gamma_2)$ are estimated from below by $\geq Ct$. \\
	\indent
	We finally note that
	$$\Vert \nabla f \Vert_1\geq\int_{\mathbb{T}}\int_{\mathbb{R}}\vert \partial_{v}f(x, v)\vert dvdx $$
	and that by the fundamental theorem of calculus, $\int_{\mathbb{R}}\vert \partial_{v}f(x, v)\vert dv$ is bounded from below by the number of alternating intersections that the line $\{x\}\times\mathbb{R}$
	makes with $\gamma_1$ and $\gamma_2$. Moreover, it can be shown by using the intermediate value theorem that this quantity is simply bounded also by $\geq ct$. This completes the proof.
	\section{Some Results for Particular Cases}
	So far, we have assumed the existence of stationary solutions, whose existence is not clear. Thus, in this section, we consider some special cases of charge distributions to gain it. We will discuss these in two cases: the case with uniform background charge and the case with a background charge distributed close to a Dirac delta distribution.
	\subsection{Explicit Formula for the Stationary State in a Uniform Background Charge case}
	We first look at the situation when the background charge is distributed uniformly in space. In this circumstance, many stable stationary solutions are available. 
	By the results above, we know that any $f^*$ of the form
	$$f^*(x, v)=\varphi\left(\frac{v^2}{2}+U^*(x)\right)$$
	is stable in twisting. Here, $$U^*(x)=\int_{0}^{x}\int_{0}^{1}K(z, y)\left[1-\int_{-\infty}^{\infty}f^{*}(y, v)dv\right]dydz.$$
	We know that the solution is stationary if $U^*(x)$ is constant. This occurs if
	$$\int_{\infty}^{\infty}f^*(y, v)dv=1$$
	which implies that $U^*(x)=0$. If we plug this in to the expression of $f^*$, we get $$f^*(x, v)=\varphi\left(\frac{v^2}{2}\right).$$ 
	Everything is consistent if we only assume that
	$$\int_{\infty}^{\infty}\varphi\left(\frac{v^2}{2}\right)dv=1.$$
	Thus, we conclude that there are many stationary solutions stable in twisting in this case. 
	\subsection{The Desingularization of Delta Distribution}
	 In this section, we show the existence of stationary solutions for the Vlasov-Poisson system on $\mathbb{T}$ with a distribution close to the Dirac delta as its background charge.\\
	More precisely, we consider the case when the ion density is given as follows:
	\begin{flalign*}
		\rho_{\epsilon}(x)=
		\begin{cases}
			\frac{1}{\epsilon} & x\in[0, \frac{\epsilon}{2})\cup (1-\frac{\epsilon}{2}, 1]\\
			0 & \text{otherwise}.
		\end{cases}
	\end{flalign*}
	Here, $\epsilon$ is a positive constant sufficiently small. Note that the formal limit $\epsilon\to0^+$ corresponds to the Dirac delta case. 
	We prove the existence of stationary solutions for the case above when $\epsilon>0$ is sufficiently small. 
	\begin{theorem}
		Consider the following Vlasov-Poisson system on a torus:
		\begin{flalign*}
			&\partial_{t}f(x, v, t)+v\cdot\partial_{x}f(x, v, t)-E(x,t)\cdot\partial_{v}f(x, v, t)=0\\
			&E(x, t)=\int_{0}^{1}K(x, y)\left[\rho_{\epsilon}(y)-\int_{-\infty}^{\infty}f(y, v, t)dv\right]dy
		\end{flalign*}
		where $K$ and $\rho_{\epsilon}$ are given as before. Then, there exists $\epsilon_0>0$ such that for all $\epsilon\in(0, \epsilon_0)$, the stationary solution for the equation above exists.
	\end{theorem}
	To prove this theorem, as in the uniform charge case, we first assume that the stationary distribution function $f$ has the following form:
	$$f(x, v)=\varphi\left(\frac{v^2}{2}+U(x)\right)$$ 
	where $U(x)$ is the electric potential satisfying 
	$$E(x)=\frac{d}{dx}U(x).$$
	In addition, we make an ansatz for the function $\varphi$:
	\begin{flalign*}
		\varphi(E)=
		\begin{cases}
			(c-E)^{\frac{5}{2}} & E<c\\
			0 & \text{otherwise}
		\end{cases}
	\end{flalign*}
	Then, the stationary Vlasov-Poisson system is equivalent to the following Poisson equation:
	\begin{flalign*}
		\frac{d^2}{dx^2}U(x)&=\rho_{\epsilon}(x)-\int_{-\infty}^{\infty}\varphi\left(\frac{v^2}{2}+U(x)\right)dv
	\end{flalign*}
	We further assume that $c>0$, $U(0)=0$, and $U(x)\geq 0$ for all $x\in\mathbb{T}$. We will show that the assumption $U(x)\geq 0$ is consistent later. Then, the equation becomes:
	\begin{flalign*}
			\frac{d^2}{dx^2}U(x)&=\rho_{\epsilon}(x)-\int_{-\sqrt{2(c-U(x))}}^{\sqrt{2(c-U(x))}}\left(c-U(x)-\frac{v^2}{2}\right)^{\frac{5}{2}}dv=\rho_{\epsilon}(x)-\sqrt{2}\left( c-U(x) \right)^3\int_{-1}^{1} \left(1-u^2\right)^{\frac{5}{2}}du\\&=\rho_{\epsilon}(x)-\frac{5\sqrt{2}\pi}{16}(c-U(x))^3.
	\end{flalign*}
	We assume that the potential satisfies $U(x)=U(1-x)$. Then, we obtain the following conditions: $$U'(0)=U'\left(\frac{1}{2}\right)=0.$$
	We now aim to reduce this equation to equations involving some parameters.
	The first step is to integrate the equations in two regions $(0, \frac{\epsilon}{2})$ and $(\frac{\epsilon}{2}, \frac{1}{2})$ separately. First, we do this in the region $(0, \frac{\epsilon}{2})$. The equation in this case is:
	\begin{flalign*}
		\frac{d^2}{dx^2}U(x)=\frac{1}{\epsilon}-\frac{5\sqrt{2}\pi}{16}\left(c-U(x)\right)^3.
	\end{flalign*}
	Multiplying $U'(x)$ on both sides and integrating in $x$ becomes
	\begin{flalign*}
		&\frac{\left(U'(x)\right)^2}{2}-\frac{U(x)}{\epsilon}-\frac{5\sqrt{2}\pi}{64}(c-U(x))^4=\frac{\left(U'(0)\right)^2}{2}-\frac{U(0)}{\epsilon}-\frac{5\sqrt{2}\pi}{64}(c-U(0))^4=-\frac{5\sqrt{2}\pi}{64}c^4
	\end{flalign*}
	when $x\in[0, \frac{\epsilon}{2}]$. Similar calculation gives
	\begin{flalign*}
		&\frac{\left(U'(x)\right)^2}{2}-\frac{5\sqrt{2}\pi}{64}(c-U(x))^4=-\frac{5\sqrt{2}\pi}{64}\left(c-U\left(\frac{1}{2}\right)\right)^4 
	\end{flalign*}
	when $x\in[\frac{\epsilon}{2}, \frac{1}{2}]$.\\
	Now, putting $x=\frac{\epsilon}{2}$ gives
	\begin{flalign*}
		\frac{U(\frac{\epsilon}{2})}{\epsilon}-\frac{5\sqrt{2}\pi}{64}c^4=-\frac{5\sqrt{2}\pi}{64}\left(c-U\left(\frac{1}{2}\right)\right)^4
	\end{flalign*}
	and hence the following relation:
	\begin{flalign*}
		U\left(\frac{1}{2}\right)=c-\left( c^4-\frac{64}{5\sqrt{2}\pi \epsilon}U\left( \frac{\epsilon}{2} \right)  \right)^{\frac{1}{4}}.
	\end{flalign*}
	Note that $U\geq0$.
	
	Now, for $x\in [0, \frac{\epsilon}{2}]$, we have
	\begin{flalign*}
		1=\frac{U'(x)}{\sqrt{\frac{5\sqrt{2}\pi}{32}(-c^4+(c-U(x))^4)+\frac{2U(x)}{\epsilon}}}.
	\end{flalign*}
	Integrating in $x$ from 0 to $\frac{\epsilon}{2}$ gives
	\begin{flalign*}
		\frac{\epsilon}{2}=\int_{0}^{U(\frac{\epsilon}{2})}\frac{dy}{\sqrt{\frac{5\sqrt{2}\pi}{32}(-c^4+(c-y)^4)+\frac{2y}{\epsilon}}}.
	\end{flalign*}
	Similarly, we have
	\begin{flalign*}
		\frac{1-\epsilon}{2}&=\int_{U(\frac{\epsilon}{2})}^{U(\frac{1}{2})}\frac{dy}{\sqrt{\frac{5\sqrt{2}\pi}{32}((c-y)^4-(c-U(\frac{1}{2}))^4)}}=\int_{U(\frac{\epsilon}{2})}^{c-\left( c^4-\frac{64}{5\sqrt{2}\pi \epsilon}U\left( \frac{\epsilon}{2} \right)  \right)^{\frac{1}{4}}}\frac{dy}{\sqrt{\frac{5\sqrt{2}\pi}{32}((c-y)^4-c^4)+\frac{2U(\frac{\epsilon}{2})}{\epsilon}}}.
	\end{flalign*}
	Now, we make a change of variables. We set $h=\frac{U(\frac{\epsilon}{2})}{\epsilon}$. Then, we can rewrite the equation as follows:
	\begin{flalign*}
		\begin{cases}
			\frac{1}{2}=\int_{0}^{h} \frac{dz}{\sqrt{\frac{5\sqrt{2}\pi}{32}(-c^4+(c-\epsilon z)^4)+2z}}:=f_1(h, c, \epsilon)\\
			\frac{1}{2}=\frac{\epsilon}{2}+\int_{\epsilon h}^{c-\left(c^4-\frac{64h}{5\sqrt{2}\pi} \right)^{\frac{1}{4}}}\frac{dw}{\sqrt{\frac{5\sqrt{2}\pi}{32}((c-w)^4-c^4)+2h}}:=f_2(h, c, \epsilon).
		\end{cases}
	\end{flalign*}	
	Thus, we have reduced the problem for finding $h\,, c>0$ satisfying the above two equations for given $\epsilon>0$. We first do this when $\epsilon=0$.
	\begin{proposition}
		There exists $h\,, c>0$ satisfying $$f_1(h, c, 0)=0\quad\text{and}\quad f_2(h, c, 0)=0.$$
	\end{proposition}
	\begin{proof}
		We have from the first equation that $$\frac{1}{2}=\int_0^h \frac{dz}{\sqrt{2z}}=\sqrt{2h}$$ and thus $h=\frac{1}{8}$.
		Plugging this into the second equation gives
		$$\frac{1}{2}=\int_{0}^{c-\left(c^4-\frac{8}{5\sqrt{2}\pi}\right)^{\frac{1}{4}}}
		\frac{dw}{\sqrt{\frac{5\sqrt{2}\pi}{32}((c-w)^4-c^4)+\frac{1}{4}}}$$
		which is equivalent to 
		$$1=\int_{\left(c^4-\frac{8}{5\sqrt{2}\pi}\right)^{\frac{1}{4}}}^{c}
		\frac{dw}{\sqrt{\frac{5\sqrt{2}\pi}{128}(w^4-c^4)+\frac{1}{16}}}$$
		We set the value on the right-hand side by $g(c)$, which is a continuous function in $c$. We show the existence of $c>0$ by showing that $\lim_{c\to\left(\frac{8}{5\sqrt{2}\pi}\right)^{\frac{1}{4}}}g(c)=\infty$ and $\lim_{c\to\infty}g(c)=0$.\\
		Indeed, we have
		\begin{flalign*}
			g(c)&=\sqrt{\frac{128}{5\sqrt{2}\pi}}\left(c^4-\frac{8}{5\sqrt{2}\pi}\right)^{-\frac{1}{4}}\int_{1}^{\left(1-\frac{8}{5\sqrt{2}\pi c^4}\right)^{-\frac{1}{4}}}\frac{dz}{\sqrt{z^4-1}}.\\
		\end{flalign*}
		Since $$\lim_{c\to\left(\frac{8}{5\sqrt{2}\pi}\right)^{\frac{1}{4}}}\int_{1}^{\left(1-\frac{8}{5\sqrt{2}\pi c^4}\right)^{-\frac{1}{4}}}\frac{dz}{\sqrt{z^4-1}}=\int_{1}^{\infty}\frac{dz}{\sqrt{z^4-1}}<\infty \quad\text{and}\quad\lim_{c\to\left(\frac{8}{5\sqrt{2}\pi}\right)^{\frac{1}{4}}}\left(c^4-\frac{8}{5\sqrt{2}\pi}\right)^{-\frac{1}{4}}=\infty$$ we have
		$$\lim_{c\to\left(\frac{8}{5\sqrt{2}\pi}\right)^{\frac{1}{4}}}g(c)=\infty.$$
		Also, we have
		$$g(c)\leq\sqrt{\frac{128}{5\sqrt{2}\pi}}\frac{1}{c}\int_{1}^{\infty}\frac{dz}{\sqrt{z^4-1}}$$
		and thus $\lim_{c\to\infty}g(c)=0.$ Therefore, there exists $c_0>0$ such that $g(c_0)=1$.
	\end{proof}
	Next, we get the existence of such a pair for small $\epsilon>0$.
	\begin{proposition}
		There exists some $\epsilon_0>0$ such that for $0<\epsilon<\epsilon_0$, there exists a pair of positive numbers $(h, c)$ satisfying 
		$$f_1(h, c, \epsilon)=0, \quad f_2(h, c, \epsilon)=0.$$
	\end{proposition}
	\begin{proof}
		To prove the existence of the pair $(h, c)$ for $\epsilon$ near 0, we use the implicit function theorem. We should check that the Jacobian does not vanish. We calculate the coefficients as follows.
		\begin{flalign*}
			\left.\frac{\partial}{\partial h}f_1\right\vert_{(\frac{1}{8}, c_0, 0)}=2,  \qquad \left.\frac{\partial}{\partial c}f_1\right\vert_{(\frac{1}{8}, c_0, 0)}=0
		\end{flalign*}
		\begin{flalign*}
			\frac{\partial}{\partial c}f_2|_{(\frac{1}{8}, c_0, 0)}&=\left.\frac{\partial}{\partial c}\left\{ \int_{0}^{c-\left(c^4-\frac{8}{5\sqrt{2}\pi} \right)^{\frac{1}{4}}}\frac{dw}{\sqrt{\frac{5\sqrt{2}\pi}{32}((c-w)^4-c^4)+\frac{1}{4}}} \right\}\right\vert_{c=c_0}
			\\&=\left.\frac{\partial}{\partial c}\left\{ \sqrt{\frac{32}{5\sqrt{2}\pi}}\left(c^4-\frac{8}{5\sqrt{2}\pi}\right)^{-\frac{1}{4}}\int_{1}^{\left(1-\frac{8}{5\sqrt{2}\pi c^4}\right)^{-\frac{1}{4}}}\frac{dz}{\sqrt{z^4-1}} \right\}\right\vert_{c=c_0}<0.
		\end{flalign*}
		\indent
		The Jacobian determinant is 
		\begin{flalign*}
			\left.\det\begin{bmatrix}
				\frac{\partial f_1}{\partial h} & \frac{\partial f_1}{\partial c}\\
				\frac{\partial f_2}{\partial h} & \frac{\partial f_2}{\partial c}
			\end{bmatrix}\right\vert_{(\frac{1}{8}, c_0, 0)} <0.
		\end{flalign*}
		\indent
		Therefore, by the implicit function theorem, there is some $\epsilon_0>0$ such that for all $0<\epsilon<\epsilon_0$, the solution 
		$$f_1(h, c, \epsilon)=0\quad f_2(h, c, \epsilon)=0$$
		exists and satisfies $h>0$, $c>0$ due to the continuous dependence on $\epsilon$. This proves the proposition.
	\end{proof}
	Since we only know that the existence of such a pair is a necessary condition for the existence of a steady state solution, we shall prove that the solution pair really corresponds to the steady state solution. We will show that this is indeed true.
	\begin{proposition}
		Suppose that $0<\epsilon<\epsilon_0$. Let $(h_\epsilon, c_\epsilon)$ be the solution for the equation obtained above. Define $U_{\epsilon}(x)$ as follows:
		\begin{flalign*}
			&x=\int_{0}^{U_\epsilon(x)}\frac{dy}{\sqrt{{\frac{5\sqrt{2}\pi}{32}(-c_\epsilon^4+(c_\epsilon-y)^4)+\frac{2y}{\epsilon}}}} \quad \text{when } 0<x<\frac{\epsilon}{2}
			\\&x-\frac{\epsilon}{2}=\int_{\epsilon h_{\epsilon}}^{U_{\epsilon}(x)}\frac{dy}{\sqrt{\frac{5\sqrt{2}\pi}{32}(-c_\epsilon^4+(c_\epsilon-y)^4)+2h_\epsilon}} \quad\text{when } \frac{\epsilon}{2}\leq x\leq \frac{1}{2}.
		\end{flalign*} 
		Then, the function $U_{\epsilon}$ is well defined, nonnegative, and is a solution for the stationary Vlasov Poisson equation corresponding to ion distribution $\rho_{\epsilon}(x)$.	
	\end{proposition}
		\begin{proof}
			That $U_\epsilon$ is well defined, continuous, and piecewise $C^2$ follows from the fact that the integrand is nonnegative, $(h_\epsilon, c_\epsilon)$ is satisfies $f_i(h_\epsilon, c_\epsilon, \epsilon)=0$ for $i=1,2$, and the inverse function theorem. $U_{\epsilon}$ is clearly nonpositive since the integrand is nonnegative and $h_{\epsilon}>0$. The fact that $U_{\epsilon}$ satisfies the Vlasov-Poisson equation follows by differentiating both sides and comparing them.
		\end{proof}
	\indent
	We show below the steady state distribution when the background charge density is ``Dirac delta" centered at the origin, i.e., $\epsilon\to0$. This result is obtained numerically by using Mathematica.\\\\
	\noindent
	\begin{figure}[H]
		\centering
		\includegraphics[scale=0.5]{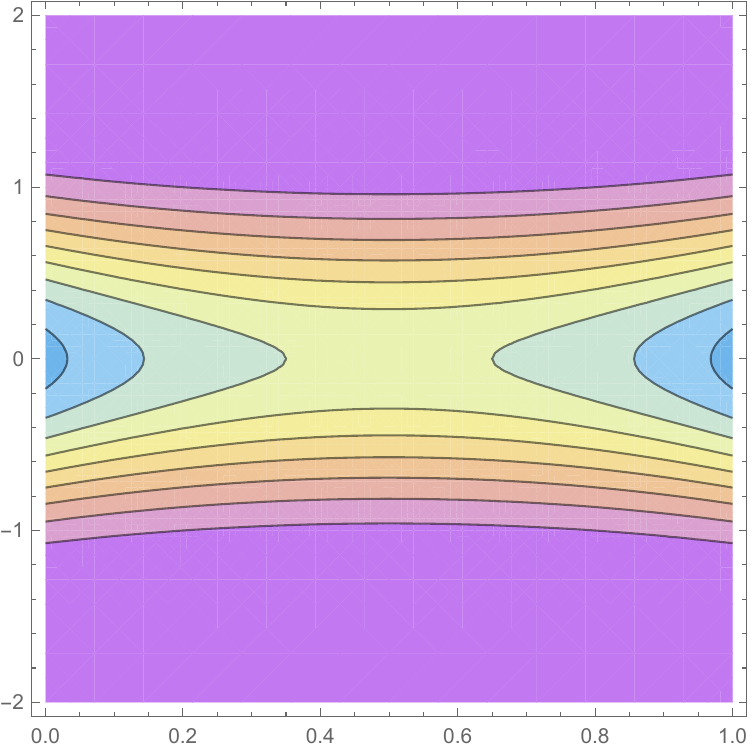} 
		\includegraphics[scale=0.5]{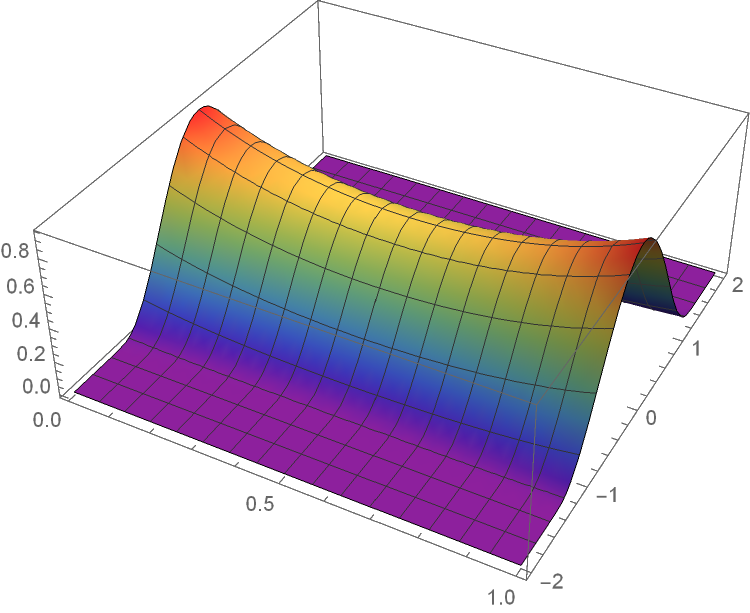}
		\caption{The graph of the steady state distribution in the case of a "Dirac Delta" background charge. The left figure represents the level sets of the distribution, where the $x$ and $y$ axes correspond to the $x$ and $v$ coordinates, respectively. The right figure is the graph of the distribution. Note that there is a discontinuity of derivatives at $x=0$ and $x=1$, which is due to the singularity of the ion distribution there.}
	\end{figure}
	
	\section*{Acknowledgement}
	
	The author would like to express his gratitude and thanks to Professor In-Jee Jeong for his guidance throughout this project. This study was supported by the Samsung Science and Technology Foundation under Project Number SSTF-BA2002-04.

\end{document}